\documentclass[11pt]{article}
\usepackage{amssymb, amsmath, amsthm, graphicx, mathrsfs}

\newtheorem{thrm}{Theorem}
\newtheorem{definition}{Definition}

\title{A note on the first-order flexes of smooth surfaces\\ which are tangent to the set of all nonrigid surfaces}
\author{Victor Alexandrov \\ \textit{Sobolev Institute of Mathematics, Novosibirsk, Russia} \\ \texttt{alex@math.nsc.ru}}

\begin{document}
\maketitle

\begin{abstract}
We prove that first-order flexes of smooth surfaces in Euclidean 3-space, 
which are tangent to the set of all nonrigid surfaces,
can be extended to second-order flexes.
\end{abstract}

\section{Introduction}\label{sec1}

The notion of an infinitesimal flex of a smooth surface in $\mathbb{R}^3$ is classical in the theory of surfaces
(see, for example, \cite{Re32}, \cite{Ef48}, and references therein).
It is useful also for the study of polyhedral surfaces and frameworks
(see, for example, \cite{AR79}, \cite{CW96}, and references given there).

In \cite{Al20}, we suggested to consider first-order flexes of polyhedral surfaces subject to the 
additional  condition ``to be tangent to a subset of the configuration space consisting of all nonrigid 
polyhedral surfaces combinatorially equivalent to the polyhedral surface under study.''
This condition is of interest because, on the one hand, it narrows the set of first-order flexes and, on the other hand, 
it holds true for any flexible polyhedral surface.
In \cite{GHT21}, S.~J. Gortler, M. Holmes-Cerfon, and L. Theran proved that \cite{Al20} leads to a novel interpretation 
of the known sufficient condition for rigidity of frameworks called ``prestress stability.''

Let us explain the essence of this new condition by the example of a polyhedral surface~$P$,
shown on the left-hand side of Fig.~\ref{fig1}, obtained by an additional triangulation
of one of the faces of nondegenerate tetrahedron $T\subset\mathbb{R}^3$.
A nontrivial first-order flex of~$P$ is shown schematically in the central part of Fig.~\ref{fig1}, where the 
red arrow is perpendicular to the additionally triangulated face of~$T$ and depicts the velocity vector of 
the corresponding vertex under the first-order flex; the velocities of the remaining vertices are equal to zero.
On the right-hand side of Fig.~\ref{fig1}, the green arrow lies in the plane of the additionally triangulated face 
of~$T$ and represents the velocity of the corresponding vertex under an infinitesimal deformation, which is a 
tangent vector to the set of all first-order nonrigid polyhedral surfaces combinatorially equivalent to~$P$.
The velocities of the remaining vertices are again equal to zero.
It is clear from Fig.~\ref{fig1} that the requirements ``to be an infinitesimal flex of a polyhedral surface'' and 
``to be an infinitesimal deformation tangent to the set of all nonrigid triangulated polyhedral surfaces combinatorially 
equivalent to the given one'' describe completely different infinitesimal deformations.
\begin{figure}
\begin{center}
\includegraphics[width=0.25\textwidth]{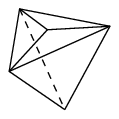}
\qquad
\includegraphics[width=0.25\textwidth]{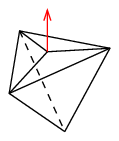}
\qquad
\includegraphics[width=0.25\textwidth]{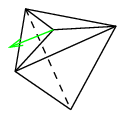}
\end{center}
\caption{\textit{On left}: Polyhedral surface~$P$ that is obtained from a tetrahedron $T$ by additional triangulation of one of its faces. \textit{In the center}: Red arrow is perpendicular to the additionally triangulated face of $T$. It represents a non-trivial first-order flex of~$P$. The velocities of all other vertices of~$P$ are equal to zero. \textit{On right}: Green arrow lies in the additionally triangulated face of $T$. It represents an infinitesimal deformation of~$P$ which is tangent to the set of all first-order nonrigid polyhedral surfaces. The velocities of all other vertices are equal to zero.}\label{fig1}
\end{figure}

In this note, we show that if a smooth surface in $\mathbb{R}^3$ is a smooth point of the set of nonrigid surfaces, then every
its first-order flex tangent to the set of nonrigid surfaces can be extended to a second-order flex.

\section{Definitions and notation}\label{sec2}

Following \cite{Re32} and \cite{Ef48}, we recall the standard definition of a higher-order flex of a smooth surface.

\begin{definition}\label{def_1}
\emph{Let $S$ be a smooth boundary free surface in $\mathbb R^3$ with position vector $\boldsymbol{x}$, 
$n$ be a positive integer,
and $\boldsymbol{\xi}^{(1)}$, $\boldsymbol{\xi}^{(2)}$, \dots, $\boldsymbol{\xi}^{(n)}$ be smooth vector fields on $S$. 
We say that the deformation of $S$ given by the formula
\begin{equation}\label{eqn_2:1}
\boldsymbol{x}_t=\boldsymbol{x}+2t\boldsymbol{\xi}^{(1)}+2t^2\boldsymbol{\xi}^{(2)}+\dots+2t^n\boldsymbol{\xi}^{(n)}
\end{equation}
is an \textit{$n$th-order flex} of $S$ 
if the change in the length of any smooth curve on $S$ is $o(t^n)$ as $t\to 0$.}
\end{definition}

Definition \ref{def_1} leads to the following relation for the first quadratic forms
$d\boldsymbol{x}_t^2-d\boldsymbol{x}^2= o(t^n)$
or, equivalently, for differentials 
$(d\boldsymbol{x}_t-d\boldsymbol{x})\,(d\boldsymbol{x}_t+d\boldsymbol{x}) = o(t^n)$.
Replacing $\boldsymbol{x}_t$ according to (\ref{eqn_2:1}), we get
\begin{equation*}
d\bigl(t\boldsymbol{\xi}^{(1)}+t^2\boldsymbol{\xi}^{(2)}+\dots+t^n\boldsymbol{\xi}^{(n)}\bigr)\,
d\bigl(\boldsymbol{x}+t\boldsymbol{\xi}^{(1)}+t^2\boldsymbol{\xi}^{(2)}+\dots+t^n\boldsymbol{\xi}^{(n)}\bigr) = o(t^n)
\end{equation*}
and, thus, obtain the following equations 
\begin{eqnarray}
d\boldsymbol{x}\, d\boldsymbol{\xi}^{(1)} &=&0,\label{eqn_2:2}\\
d\boldsymbol{x}\, d\boldsymbol{\xi}^{(2)}+ d\boldsymbol{\xi}^{(1)}\,d\boldsymbol{\xi}^{(1)}&=&0,\label{eqn_2:3}\\
\dots \quad \dots \quad \dots & & \dots \nonumber \\
d\boldsymbol{x}\, d\boldsymbol{\xi}^{(n)}+\sum_{j=1}^{n-1}d\boldsymbol{\xi}^{(j)}\,d\boldsymbol{\xi}^{(n-j)} &=&0.\label{eqn_2:4}
\end{eqnarray}

Note that the presence of factors 2 in expression (\ref{eqn_2:1}) simplifies equations (\ref{eqn_2:2})--(\ref{eqn_2:4}).
This simplification was proposed by E. Rembs in \cite{Re32} in 1932.
Since that time, expression (\ref{eqn_2:1}) has been the standard notation for the $n$th-order flex.

Equation (\ref{eqn_2:2}) means that $\boldsymbol{x}+2t\boldsymbol{\xi}^{(1)}$ is a first-order flex of $S$.
In local coordinates $u,v$ on $S$, (\ref{eqn_2:2}) is equivalent to 
$(\boldsymbol{x}_udu+\boldsymbol{x}_vdv)\, (\boldsymbol{\xi}_u^{(1)}du+\boldsymbol{\xi}_v^{(1)}dv) =0$
and, thus, to the following system of three partial differential equations
\begin{equation}\label{eqn2:5}
\boldsymbol{x}_u \cdot \boldsymbol{\xi}_u^{(1)} =0,\quad
\boldsymbol{x}_u \cdot \boldsymbol{\xi}_v^{(1)} +\boldsymbol{x}_v \cdot \boldsymbol{\xi}_u^{(1)} =0,  \quad
\boldsymbol{x}_v \cdot \boldsymbol{\xi}_v^{(1)} =0, 
\end{equation}
where $\cdot$ stands for the scalar product in $\mathbb{R}^3$.
A first-order flex $\boldsymbol{x}+2t\boldsymbol{\xi}^{(1)}$ of $S$ is called \textit{trivial} if it is
generated by a smooth family of isometries of $\mathbb{R}^3$.
$S$ is called \textit{first-order rigid} if every its first-order flex is trivial;
otherwise, $S$ is called \textit{first-order nonrigid}.

Equations (\ref{eqn_2:2}) and (\ref{eqn_2:3}) mean that $\boldsymbol{x}+2t\boldsymbol{\xi}^{(1)}+2t^2\boldsymbol{\xi}^{(2)}$ 
is a second-order flex of $S$.
It is called an \textit{extension} of the first-order flex $\boldsymbol{x}+2t\boldsymbol{\xi}^{(1)}$.

\begin{definition}\label{def_2}
\emph{Let $S$ be a smooth boundary free surface in $\mathbb R^3$ with position vector $\boldsymbol{x}$.
We say that a first-order flex $\boldsymbol{x}+2t\boldsymbol{\xi}^{(1)}$ of $S$ is \textit{tangent} 
to the set of all nonrigid smooth surfaces if the following conditions hold true:}

\emph{(i) there is a smooth family $\{S(r)\}_{r\in(-1,1)}$ of boundary free nonrigid smooth surfaces in $\mathbb R^3$
such that $S(0)=S$, i.\,e., $\boldsymbol{x}(0)=\boldsymbol{x}$, where $\boldsymbol{x}(r)$ is position vector of $S(r)$;}

\emph{(ii) there is a smooth family $\{2\boldsymbol{\xi}^{(1)}(r)\}_{r\in(-1,1)}$ of vector fields such that,
for every $r\in (-1,1)$, $\boldsymbol{x}(r)+2t\boldsymbol{\xi}^{(1)}(r)$ is a first-order flex of $S(r)$  and}
\begin{equation}\label{eqn_2:6}
\frac{d}{dr}\biggl\vert_{r=0}\biggr. \boldsymbol{x}(r)=2\boldsymbol{\xi}^{(1)}(0)\qquad\mbox{\emph{and}}\qquad 
\boldsymbol{\xi}^{(1)}(0)=\boldsymbol{\xi}^{(1)}.
\end{equation}
\end{definition}

Conditions (i) and (ii) of Definition~\ref{def_2} mean that $S$ lies on the curve 
$\{S(r)\}_{r\in(-1,1)}$, located in the set of nonrigid surfaces, and 
$2\boldsymbol{\xi}^{(1)}$ is the velocity vector of the point $\boldsymbol{x}(r)$ moving along this curve at $S$.
That is, Definition~\ref {def_2} is consistent with the standard for classical differential geometry 
point of view on the tangent vector to a surface as the velocity vector of a point moving along the curve 
lying on the surface.

\section{Main result}\label{sec3}

\begin{thrm}\label{th_1}
Let $S$ be a smooth boundary free surface in $\mathbb R^3$ with position vector $\boldsymbol{x}$.
And let $\boldsymbol{x}+2t\boldsymbol{\xi}^{(1)}$ be a first-order flex of $S$,
which is tangent to the set of all nonrigid smooth surfaces.
Then the first-order flex $\boldsymbol{x}+2t\boldsymbol{\xi}^{(1)}$ can be extended to 
a second-order flex of~$S$.
\end{thrm}

\begin{proof}
Let $\{S(r)\}_{r\in(-1,1)}$ and $\{2\boldsymbol{\xi}^{(1)}(r)\}_{r\in(-1,1)}$  be the smooth families from 
Definition~\ref{def_2}.
Since, for every $r\in (-1,1)$, $\boldsymbol{x}(r)+2t\boldsymbol{\xi}^{(1)}(r)$ is a first-order flex of $\boldsymbol{x}(r)$, 
we have $d\boldsymbol{x}(r)\,d\boldsymbol{\xi}^{(1)}(r)=0$.
Differentiating this equality at $r=0$ and taking into account (\ref{eqn_2:6}), we get 
\begin{equation}\label{eqn_3:1}
0 = \dfrac{d}{dr}\biggl\vert_{r=0}\biggr. \bigl[d\boldsymbol{x}(r)\,d\boldsymbol{\xi}^{(1)}(r)\bigr] = 
d\biggl[\dfrac{d}{dr}\biggl\vert_{r=0}\boldsymbol{x}(r)\biggr]\, d\boldsymbol{\xi}^{(1)}(0) 
+ d\boldsymbol{x}(0)\,d\biggl[\dfrac{d}{dr}\biggl\vert_{r=0}\boldsymbol{\xi}^{(1)}(r)\biggr] = 
2\bigl[d\boldsymbol{\xi}^{(1)}\, d\boldsymbol{\xi}^{(1)} + d\boldsymbol{x}\,d\boldsymbol{\xi}^{(2)}\bigr],
\end{equation}
where we have putten by definition
\begin{equation*}
2\boldsymbol{\xi}^{(2)}=\dfrac{d}{dr}\biggl\vert_{r=0}\boldsymbol{\xi}^{(1)}(r).
\end{equation*}
It follows from (\ref{eqn_3:1}) that 
$d\boldsymbol{x}\,d\boldsymbol{\xi}^{(2)} + d\boldsymbol{\xi}^{(1)}\, d\boldsymbol{\xi}^{(1)} =0$ 
and, thus, the deformation $\boldsymbol{x}+2t\boldsymbol{\xi}^{(1)}+2t^2\boldsymbol{\xi}^{(2)}$ 
is a second-order flex of $S$.
On the other hand, $\boldsymbol{x}+2t\boldsymbol{\xi}^{(1)}+2t^2\boldsymbol{\xi}^{(2)}$ is obviously an extension of 
the first-order flex $\boldsymbol{x}+2t\boldsymbol{\xi}^{(1)}$.
\end{proof}

Theorem \ref{th_1} shows that, for a smooth surface $S$ in $\mathbb{R}^3$ which is a smooth point of the set $\mathscr{S}$ 
of all nonrigid surfaces, the condition that its first-order flex is tangent to $\mathscr{S}$ implies that this
first-order flex can be extended to a second-order flex.
We cannot prove a similar statement in the case when $S$ is not a smooth point of $\mathscr{S}$, 
since we know nothing about the structure of $\mathscr{S}$.
For example, we do not know what its dimension (or codimension) is, nor what is the stucture of the set of its nonsmooth points.
We can only hope that these issues will be clarified in the future.

\section{Acknowledgment}\label{sec5}

The study was carried out within the framework of the state contract
of the Sobolev Institute of Mathematics (project no. 0314-2019-0006).

\end{document}